\documentclass[12pt]{article}
\usepackage[left=2cm,right=2cm,top=1.5cm,bottom=1.5cm]{geometry}
\usepackage{amsfonts}
\usepackage{amssymb}
\usepackage{amsmath, amsthm}
\usepackage{eufrak}
\usepackage[dvips]{graphicx}
\usepackage{latexsym}
\usepackage{color}

\DeclareMathSymbol{\Z}{\mathbin}{AMSb}{"5A}
\DeclareMathSymbol{\N}{\mathbin}{AMSb}{"4E}
\DeclareMathSymbol{\Q}{\mathbin}{AMSb}{"51}

\newcommand{\bi}{\begin{itemize}}
\newcommand{\ei}{\end{itemize}}

\newtheorem{theorem}{Theorem}[section]

\newtheorem{fact}[theorem]{Fact}
\newtheorem{claim}[theorem]{Claim}
\newtheorem{corollary}[theorem]{Corollary}
\newtheorem{proposition}[theorem]{Proposition}
\newtheorem{definition}[theorem]{Definition}
\newtheorem{remark}[theorem]{Remark}

\newtheorem{question}[theorem]{Question}

\newtheorem{problem}[theorem]{Problem}
\newtheorem{example}[theorem]{Example}

\newtheorem{thm}{Theorem}[section]

\newtheorem{prop}[thm]{Proposition}

\newcommand{\be}{\begin{enumerate}}
\newcommand{\ee}{\end{enumerate}}

\def\Ind#1#2{#1\setbox0=\hbox{$#1x$}\kern\wd0\hbox to 0pt{\hss$#1\mid$\hss}
\lower.9\ht0\hbox to 0pt{\hss$#1\smile$\hss}\kern\wd0}

\def\Notind#1#2{#1\setbox0=\hbox{$#1x$}\kern\wd0\hbox to 0pt{\mathchardef
\nn=12854\hss$#1\nn$\kern1.4\wd0\hss}\hbox to
0pt{\hss$#1\mid$\hss}\lower.9\ht0 \hbox to
0pt{\hss$#1\smile$\hss}\kern\wd0}

\title{Some remarks on inp-minimal and finite burden groups}

\author{Jan Dobrowolski
\footnote{The first author was supported by  Samsung Science Technology Foundation
under Project Number
SSTF-BA1301-03, and by European Union's Horizon 2020 research
and innovation programme under the Marie Sklodowska-Curie grant agreement No 705410} 
and John Goodrick}
\date{}

\begin{document}

\setcounter{section}{-1}
\maketitle
\begin{abstract}
We prove that any left-ordered inp-minimal group is abelian and we provide an example of a non-abelian left-ordered group of dp-rank 2. Furthermore, we establish a necessary condition for group to have finite burden involving normalizers of definable sets, reminiscent of other chain conditions for stable groups.
\end{abstract}

\section{Introduction and preliminaries}
One of the model-theoretic properties that gained a lot of interest recently is dp-minimality, which, on one hand, significantly strengthens NIP, and on the other hand, is satisfied by all strongly minimal theories, all (weakly) o-minimal theories, algebraically 
closed valued fields (more generally, by all C-minimal structures), and the valued field of p-adics.
Several interesting results were obtained for dp-minimal structures in the algebraic contexts
of groups and fields (sometimes with additional structure), see for example \cite{2,3,4,5}.


Throughout this note, we work in the context of a complete first-order theory $T$, and ``formula'' means a first-order formula in the language of $T$.

We recall some key definitions, which are originally due to Shelah \cite{Shelah}, though the precise form of the definitions which we give below seems to come from Usvyatsov \cite{Usvyatsov}.

\begin{definition}
\begin{enumerate}
\item An \emph{inp-pattern of depth $\kappa$ (in the partial type $\pi(\overline{x})$)} is a sequence $\langle \varphi_i(\overline{x}; \overline{y}_i) : i < \kappa \rangle$ of formulas and an array $\{ \overline{a}_{ij} : i < \kappa, j < \omega \}$ of parameters (from some model of $T$) such that:
\begin{enumerate}
\item For each $i < \kappa$, there is some $k_i < \omega$ such that $\{\varphi_{i,j}(\overline{x}; \overline{a}_{i, j}) : j < \omega \}$ is $k_i$-inconsistent; and
\item For each $\eta : \kappa \rightarrow \omega$, the partial type $$\pi(\overline{x}) \cup \{\varphi_i(\overline{x}; \overline{a}_{i, \eta(i)}) : i < \kappa \}$$ is consistent.
\end{enumerate}

\item The \emph{inp-rank} (or \emph{burden}) of a partial type $\pi(\overline{x})$ is the maximal $\kappa$ such that there is an inp-pattern of depth $\kappa$ in $\pi(\overline{x})$, if such a maximum exists. In case there are inp-patterns of depth $\lambda$ in $\pi(\overline{x})$ for every cardinal $\lambda < \kappa$ but no inp-pattern of depth $\kappa$, we say that the inp-rank of $\pi(\overline{x})$ is $\kappa_-$.

\item The inp-rank of $T$ is the inp-rank of $x=x$, and $T$ is \emph{inp-minimal} if its inp-rank is $1$.

\item An \emph{ict-pattern of depth $\kappa$ (in the partial type $\pi(\overline{x})$)} is a sequence $\langle \varphi_i(\overline{x}; \overline{y}_i) : i < \kappa \rangle$ of formulas and an array $\{ \overline{a}_{ij} : i < \kappa, j < \omega \}$ of parameters (from some model of $T$) such that for each $\eta : \kappa \rightarrow \omega$, the partial type $$\pi(\overline{x}) \cup \{\varphi_i(\overline{x}; \overline{a}_{i, \eta(i)}) : i < \kappa \} \cup \{\neg \varphi_i(\overline{x}; \overline{a}_{i,j}) : i < \kappa, j \neq \eta(i) \}$$ is consistent.

\item The \emph{dp-rank} of a partial type $\pi(\overline{x})$ is the maximal $\kappa$ such that there is an ict-pattern of depth $\kappa$ in $\pi(\overline{x})$, if such a maximum exists (if not, it is ``$\kappa_-$'' exactly as in 2 above). The dp-rank of $T$ is the dp-rank of $x=x$, and $T$ is \emph{dp-minimal} if its dp-rank is $1$.
\end{enumerate}

\end{definition}

In spite of its name, dp-rank is really more like a cardinal-valued \emph{dimension} than an ordinal-valued rank such as $SU(p)$, and in the context of stable theories, dp-minimality is equivalent to every nonalgebraic $1$-type having weight $1$, as observed in \cite{7}. It turns out that a theory is dp-minimal just in case it is both inp-minimal and NIP (see \cite{8}). 

One of the context investigated in \cite{2} is that of (bi)-ordered groups.

\begin{definition}
A \emph{left-ordering} on a group $(G, \cdot)$ is a total ordering $<$ on $G$ such that for any $f, g, h \in G$, 
whenever $g < h$, we have that $f \cdot g < f \cdot h$. A \emph{right-ordering} is defined similarly, and a \emph{bi-ordering}
on $G$ is an ordering which is simultaneously a left-ordering and a right-ordering.
\end{definition}

While Pierre Simon claimed that all inp-minimal ``ordered groups'' are abelian \cite{2}, 
his proof only applies to groups with a definable bi-ordering: his argument uses the fact that for any $x, y$ in a
bi-orderable group and any positive $n \in \N$, if $x^n = y^n$ then $x = y$. But in left-orderable groups (such as in the 
example of the Klein bottle group below), one may have that $x^2 = y^2$ but $x \neq y$.

The main result of Section 2 of this note is that every inp-minimal left-ordered group is abelian (Theorem~\ref{LO_abelian}), which strengthens the result mentioned above from \cite{2}. We
also show that this conclusion fails already in the dp-rank 2 case by providing a suitable example (Section 1). Finally, in Section 3 we consider necessary conditions for an arbitrary (not necessarily ordered) group to have finite burden. In the stable case, this gives a simple and apparently new condition on stable groups $G$ of finite weight: such a group must contain finitely many definable abelian subgroups $A_0, \ldots, A_k$ such that $G / N[A_0] \ldots N[A_k]$ has finite exponent (Corollary~\ref{quotient}).

\section{A non-abelian 
left-ordered group of dp-rank 2}\label{example}

In this section, we define the ``Klein bottle group'' (the fundamental group of a Klein bottle) which is presented as $G = \langle x, y : x^{-1} y x = y^{-1} \rangle$. In other words, $y^{-1} x = x y$, and routine algebraic manipulation shows:

\begin{enumerate}
\item Every $g \in G$ can be uniquely written as $g = x^n y^m$ for some $n, m \in \Z$,
\item $(x^n y^m) \cdot (x^{n'} y^{m'}) = x^{n+n'} y^{m' + (-1)^{n'} m}$, and
\item $(x^n y^m)^{-1} = x^{-n} y^{(-1)^{n+1} m}$.
\end{enumerate}

We can define a left ordering $\leq$ on $G$ lexicographically on the exponents: $x^n y^m \leq x^{n'} y^{m'}$ iff either $n < n'$ or else $n = n'$ and $m \leq m'$. The subgroup generated by $y$ is the minimal nontrivial convex subgroup of $G$, and the order type of $G$ is $\Z \times \Z$.

We note in passing that while $G$ is non-abelian, it is abelian-by-\emph{finite}: a simple calculation shows that the centralizer $C(y)$ of $y$ is $\{x^{2n} y^m : n, m \in \Z \}$, which is abelian, and for any $g \in G$, we have $g^2 \in C(y)$.

\begin{prop}
The structure $(G, \cdot, \leq)$ is dp-rank $2$.
\end{prop}

\begin{proof}

To show that it is NIP and dp-rank \emph{at most} $2$, it suffices (thanks to the additivity of the dp-rank proved in \cite{6}) to note that an isomorphic copy of $G$ is definable in the dp-minimal structure $(\Z, <, +)$ with $\Z \times \Z$ as the the universe of the group; the definition of the group operation depends on the parity of one of the coordinates, but of course $2 \Z$ is a definable subgroup of $\Z$.

Now we compute the centralizer $C(x)$ of the generator $x$. For any $a, b \in \Z$, we have $$x \cdot (x^a y^b) = x^{a+1} y^b$$ and $$(x^a y^b) \cdot x = x^{a+1} y^{-b},$$ so we conclude that $C(x) = \langle x \rangle$.

For any $n \in \N$, there are pairwise disjoint intervals $\langle I_{n,k} : k < \omega \rangle$ such that for any $k$, $x^k \in I_{n,k}$ and $I_{n,k}$ intersects every right coset $C(x), C(x) y, \ldots, C(x) y^n$. So by compactness, in an $\omega$-saturated extension of $G$, we can find pairwise disjoint intervals $\langle I_k : k < \omega \rangle$ such that $x^k \in I_k$ and each $I_k$ intersects every right coset $C(x) y^n$. Therefore the formulas expressing $z \in I_k$ and $z \in C(x) y^j$ (in the free variable $z$) give an inp-pattern of depth $2$, so $G$ is not inp-minimal, hence by NIP it is not dp-minimal.


\end{proof}

\begin{remark}
In a previous version of this paper, we asked whether the group $G$ above is dp-minimal in the pure language of groups. This was answered negatively by Halevi and Hasson \cite{HH}.
\end{remark}

\section{Inp-minimal left-ordered groups}
In this section, we prove that every inp-minimal left-ordered group is abelian.
For a left-ordered group $G$ and a subset $A\subseteq G$, by $h(A)$ we will denote
the convex hull of $A$ in $G$.
\begin{fact}[\cite{2}]\label{simon}
 Let G be an inp-minimal group. Then there is a definable normal abelian
 subgroup $H$ of $G$ such that $G/H$ has finite exponent.
\end{fact}

Note that (in contrast to the bi-ordered groups), in a left-ordered group $G$,
 the convex hull of a subgroup $H$ need not be a subgroup of $G$:
 
 \begin{example}
  Consider the left-ordered group $G=(Aut(\mathbb{Q},<),\prec)$, where $\prec$ is a standard left-order on $Aut(\mathbb{Q},<)$
  coming from a well-order on $\mathbb{Q}$, i.e., $f\prec g$ if $f(x)<g(x)$ for $x$ being the smallest (in the sense of the well-order) element
  on which $f$ differs from $g$. Choose $a<b<c<d<e\in \mathbb{Q}$, where $a$ is the first element of $\mathbb{Q}$ in the fixed
  well-order, and $f,g\in G$ such that $f(a)=c$, $f(d)=d$, $g(a)=b$ and $g(b)=e$. Then, $g^2(a)=e$, and, for any $k\in \mathbb{Z}$, $f^k(a)<e$, so
  $f^k\prec g^2$. So we have that $g^2\notin h(<f>)$, but, clearly, $g\in h(<f>)$. So we get that $h(<f>)$ is not a subgroup of $G$.
 \end{example}
Nevertheless, in left-ordered groups, $h(H)$ is always a union
 of right $H$-cosets, and, as an analogue of Lemma 3.2
 from  \cite{2}, we obtain:
 
 \begin{fact}\label{simon_left}
  Let $G$ be  an  inp-minimal  left-ordered  group.   Let $H$ be  a  definable
subgroup of $G$ and let $C$ be the convex hull of $H$.  Then $C$ is a union of finitely many
right $H$-cosets.
 \end{fact}
 {\em Proof.}
 All we need to repeat the proof from \cite{2} is to prove that all cosets of $H$ contained in $C$ are cofinal in $C$.
 So take any $c\in C$ and fix $h\in H$. Choose $h_1\in H$ such that $h_1<c$. Then, by the left-invariance of 
 the order, $h=(hh_1^{-1})h_1<hh_1^{-1}c\in Hc$, so $Hc$ is cofinal in $C$.
 
\hfill $\square$\\

We will also use a group-theoretic fact about FC-groups.
\begin{definition}
An FC-group is a group in which the centralizer of every element has finite index.
\end{definition}
Note that if $[G:Z(G)]$ is finite, then $G$ is an FC-group.
The following is Theorem 6.24 from \cite{1}:
\begin{fact}\label{lam}
 Every torsion-free FC-group is abelian.
\end{fact}

\begin{theorem}
\label{LO_abelian}
Every left-ordered inp-minimal group is abelian.
\end{theorem}
{\em Proof.}
By Fact \ref{lam}, it is enough to show that $[G:Z(G)]$ is finite.
Let $H$ be given by Fact \ref{simon}. Notice that $G=h(H)$: if $a\in G$, say $a>e$, then
for $l$ equal to the exponent of $G/H$, we have $e<a<a^2<\dots<a^l\in H$, so $a\in h(H)$.
Hence, by Fact \ref{simon_left}, we get that $n:=[G:H]$ is finite.
\begin{claim}
\label{interval}
For any positive $x\in G$, the interval $[e,x]$ is covered by finitely many right cosets 
of a central subgroup of $G$.
\end{claim}
{\em Proof of Claim~\ref{interval}.}
We can assume that $G$ is non-trivial (hence it is infinite so also $H$ is non-trivial).
Notice that  any coset $Hg$ in $G$ has a positive representative (if $g$ is negative
then one can take $g^{-(l-1)}$ as such a representative). It follows (thanks to the normality of $H$) that for any 
$y,g\in G$ we can find an element $z\in G$ such that $Hg=Hz$ and $z>y$ 
(by choosing $z:=yw$, where $w$ is a positive element of $Hy^{-1}g$).

Now, fix any positive $x=x_0\in G$. Using the above observation, we can choose 
$x_0<x_1<\dots<x_{n-1}\in G$ such that $G=\bigcup_{i<n} Hx_i$. 
By Fact \ref{simon_left}, each  set $h(C(x_i))$ is covered by finitely many right $C(x_i)$-cosets; call them $C(x_i) k_{i,0}, \ldots, C(x_i) k_{i, m(i)}$. Thus if $$C:=\bigcap_{i<n} h(C(x_i)),$$ then any $y \in C$ belongs to some intersection $$H x_j \cap \bigcap_{i < n} C(x_i) k_{i, \eta(i)}$$ for some $j < n$ and some $\eta : n \rightarrow \omega$. But the above intersection of right cosets is a right coset of $$A:=H\cap \bigcap_{i<n} C(x_i),$$ hence $C$ is covered by finitely many $A$-cosets

But, since $H$ is abelian, and $G$ is generated by $H,x_0,\dots,x_{n-1}$, we have that 
$A\subseteq Z(G)$. Also, since $x_0<x_1<\dots < x_{n-1}$ and $\forall _{i<n} x_i\in C(x_i)$,
we get that $x=x_0\in C$, so, by convexity of $C$, $[e,x_0]\subseteq C$, which proves the claim.
\hfill $\square$\\
Now, suppose for a contradiction that $[G:Z(G)]$ is infinite. Note that if some coset $Z(G) g$ contains only negative elements, then the coset $Z(G) g^{-1}$ contains only positive elements, so in any case we may choose infinitely many
positive representatives $y_0,y_1,y_2,\dots$ of pairwise distinct right cosets of
$Z(G)$ in $G$. Without loss of generality, $G$ is $\omega$-saturated, and there is an element $x\in G$ greater than all the $y_i$'s. Then $[e,x]$ cannot be covered 
by finitely many right cosets of $Z(G)$, so it cannot be covered by finitely many right cosets of 
any central subgroup of $G$, contradicting the Claim.
\hfill $\square$\\

\begin{corollary}
If $(G, \cdot, <)$ is a left-ordered group which is inp-minimal (in the pure language of ordered groups), then it is dp-minimal.
\end{corollary}

\begin{proof}
By Theorem~\ref{LO_abelian}, $G$ is abelian, and any ordered abelian group is NIP, as shown in \cite{9}; since NIP and inp-minimality imply dp-minimality, we are done.
\end{proof}

\section{Some observations on groups of finite inp-rank}

The example from Section \ref{example}, as every group definable in the Presburger artithmetic, is abelian-by-finite (see
\cite{10}). It seems natural to ask the following general question:
\begin{problem}
What can be said about ordered groups of finite inp-rank (possibly under some additional model-theoretic assumptions)?
\end{problem}

To apply some ideas from the proof of Theorem \ref{LO_abelian}, it seems necessary to prove some variant of the following
property, which was essentially observed in the proof of Proposition 3.1 from \cite{2}:

\begin{fact}
 If $G$ is an inp-minimal group and $H,K<G$ are definable, then either $[H:H\cap K]$ or
 $[K:H\cap K]$ is finite.
\end{fact}

Below, we make an observation of this kind in the context of finite inp-rank, but we need to work with normal subgroups.

If $G$ is a group and $A\subseteq G$, then by $N[A]$ we shall denote the normal subgroup of $G$ generated by $A$.
If $H$ is a subgroup of $G$, then we put $A/H:=\{aH:a\in A\}$. For any elements $g,h\in G$, by $g^h$ we mean the 
conjugate $h^{-1}gh$ of $g$ by $h$.
\begin{proposition}\label{chain}
 If $G$ is a group of burden $n\in \omega$, then there do not exist definable sets $D_0,D_1,\dots,D_n$ such
 that, if we put $N_i=N[D_i]$, then $$(\forall i\leq n) (|D_i/N_0N_1\dots N_{i-1}N_{i+1}N_{i+2}\dots N_n|\geq \omega).$$
 Moreover, we can replace the above condition by:  
 for  each $i\leq n$, there is an infinite subset $E_i$ of $D_i$ such that 
 $$(\forall e_0,e_1\in E_i)(e_0e_1^{-1}\in ((D_0D_1\dots D_{i-1}D_{i+1}D_{i+2}\dots D_n)^G)^{2n} \implies e_0=e_1).$$ 
\end{proposition}
{\em Proof.}
Clearly, it is enough to prove the ``moreover'' part. Suppose for a contradiction that there exist sets $(D_i)_{i\leq n}$
and $(E_i)_{i\leq n}$ as above. For each $i\leq n$, let $(e_{i,j})_{j<\omega}$ be a sequence of
pairwise distinct elements of $E_i$. We claim that the formulas 
$$(\phi_i(x,e_{i,j}):=x\in D_0D_1\dots D_{i-1}e_{i,j}D_{i+1}D_{i+2}\dots
D_n)_{i\leq n,j<\omega}$$ form an inp-pattern of depth $n+1$, which will contradict the assumption.
Obviously, for any $\eta\in \omega^{n+1}$, the element $\prod_{i\leq n} e_{i,\eta(i)}$ satisfies $\bigwedge_{i\leq n}
\phi_i(x,e_{i,\eta(i)})$. On the other hand, if there is some $g$ satisfying both $\phi_i(x,e_{i,j_{0}})$ and 
$\phi_i(x,e_{i,j_{1}})$, then  for some $(d_{k}, d_{k}')_{k\in \{0,1,\dots,n\}\backslash \{i\}}$ with $d_k,d_k'\in D_k$,
we have $$d_0d_1\dots d_{i-1}e_{i,j_0}d_{i+1}\dots d_n=g=d_0'd_1'\dots d_{i-1}'e_{i,j_1}d_{i+1}'\dots d_n',$$
so $$e_{i,j_0}e_{i,j_1}^{-1}=d_{i-1}^{-1}d_{i-2}^{-1}\dots d_0^{-1}d_0'd_1'\dots d_{i-1}'(d_{i+1}'d_{i+2}'
\dots d_n'd_n^{-1}d_{n-1}^{-1}\dots d_{i+1}^{-1})^{e_{i,j_1}^{-1}}$$ is an element of 
$((D_0D_1\dots D_{i-1}D_{i+1}D_{i+2}\dots D_n)^G)^{2n}$, hence, by the assumption on $E_i$, we get that
$e_{i,j_0}=e_{i,j_1}$. This completes the proof.
\hfill $\square$\\

\begin{example}
 Note that the above proposition does not follow from the (somewhat similarly looking) chain condition \cite[Proposition 4.5 (2)] {11},
  as the latter is satisfied in a non-abelian  free group $F$ (since the only non-trivial definable proper subgroups are 
  the cyclic groups which are not normal), and the conclusion of Proposition \ref{chain} is not satisfied in $F$.\\
  To see this, we may assume (as the failure of the conclusion of Proposition is $\bigwedge$-expressible)
  that $F$ is the free group on generators $x_0,x_1,x_2,\dots$. Put $D_i=\{x_i^m:m<\omega\}$.
  Then, clearly, $$(\forall i\leq n) (|D_i/N[D_0]N[D_1]\dots N[D_{i-1}]N[D_{i+1}]N[D_{i+2}]\dots N[D_n]|\geq \omega).$$
\end{example}

Using the above chain condition we get in the stable context:
\begin{corollary}\label{quotient}
 If $G$ is a stable group of finite weight, then there are finitely many definable abelian subgroups $A_0,\dots,A_k$ of $G$ such
 that the quotient $G/N[A_0]N[A_1]\dots N[A_k]$ has finite exponent.
\end{corollary}
{\em Proof.}
It follows from the assumptions that $G$ has a finite burden, say $n$. 
For any $g\in G$ put $A_g=C(C(g))$. Note that $A_g$ is a definable abelian group, containing the group generated by $g$.
Suppose for a contradiction that the conclusion
does not hold. Then, using compactness, we can find inductively a sequence $(g_i)_{i<\omega}$ of elements of $G$ such that
for each $i,m<\omega$, $g_i^m\notin N[A_{g_0}]N[A_{g_1}]\dots N[A_{g_{i-1}}]$ . Since the latter is a 
type-definable condition on the sequence $(g_i)_{i<\omega}$ (as $N[A_{g_0}]N[A_{g_1}]\dots N[A_{g_{i-1}}]$ is
$\bigvee$-definable over $g_0,g_1,\dots,g_{i-1}$), we can additionally assume that $(g_i)_{i<\omega}$ is an indiscernible
sequence. Now, by Proposition \ref{chain}, there is some $i\leq n$ such that 
for some $m<\omega$, $g_i^m\in ((A_0A_1\dots A_{i-1}A_{i+1}A_{i+2}\dots A_n)^G)^{2n}$ (otherwise, putting
$D_i=A_i$ and $E_i=\{g_i^m:m<\omega\}$, we contradict the conclusion of the Proposition).
But this is expressible by a sentence $\phi(g_i;g_0,g_1,\dots, g_{i-1},g_{i+1},g_{i+2},\dots, g_n)$,
and by the choice of the $g_i$'s, the sentence\\ $\phi(g_n;g_0,g_1,\dots,g_{n-1})$ is not true in $G$, so the sequence
$(g_i)_{i<\omega}$ is not totally indiscernible. This contradicts stability.
\hfill $\square$\\

We end by stating a question about relaxing the assumption of stability in the last theorem to some settings which
allow the existence of a definable order:

\begin{question}
 Is the conclusion of Theorem \ref{quotient} true for:\\
 1) rosy groups of finite burden? (in particular, for simple groups of finite weight and groups definable in o-minimal structures?)\\
 2) distal groups of finite burden?
 \end{question}

\noindent
{Jan Dobrowolski\\
School of Mathematics, University of Leeds, United Kingdom\\
e-mail: dobrowol@math.uni.wroc.pl\\\
John Goodrick\\
Departamento de Matem\'aticas, Universidad de los
Andes, Bogot\'a, Colombia\\
e-mail: jr.goodrick427@uniandes.edu.co


\begin{thebibliography}{99}

\bibitem{8}  H. Adler, {\em Strong theories, burden, and weight},  preprint, 2007. 
arXiv:1507.0391


\bibitem{11} A. Chernikov, I. Kaplan and P. Simon {\em Groups and fields with NTP2}, Proceedings of the AMS 143, 395-406,
2015.

\bibitem{9}  Y. Gurevich and P. Schmitt, {\em The theory of ordered abelian groups does not have the independence property}, Trans. Amer. Math. Soc., 284:171-182, 1984.
arXiv:1507.0391

\bibitem{4}  F. Jahnke, P. Simon, and E. Walsberg, {\em Dp-minimal valued fields},  2015.
arXiv:1507.0391

\bibitem{3} W. Johnson, {\em On dp-minimal fields}, 2015.
arXiv:1507.0274

\bibitem{6} I. Kaplan, A. Onshuus, and A. Usvyatsov, {\em Additivity of the dp-rank}, Trans. Amer. Math.
Soc., 365(11):5783-€"5804, 2013

\bibitem{1} T.Y. Lam, {\em A First Course in Noncommutative Rings}, Springer-Verlag, 1991.

\bibitem{5}  E. Levi, I. Kaplan, P. Simon {\em Some remarks on dp-minimal groups},  2016.
 	arXiv:1605.07867

 	
 \bibitem{10} A. Onshuus and M. Vicaria, {\em Definable groups in models of Presburger Arithmetic and $G^{00}$},
 preprint, 2016.
	
 	
\bibitem{7} A. Onshuus and A. Usvyatsov, {\em On dp-minimality, strong dependence, and weight}, J. Symb. Log. 76-3, 2011.

\bibitem{Shelah} S. Shelah, {\em Dependent first-order theories, continued}, 
Israel J. of Math., Volume 173-1, 2009.

\bibitem{2} P. Simon, {\em On dp-minimal ordered structures}, 
J. Symb. Log., Volume 76-2, 2011.

\bibitem{Usvyatsov} A. Usvyatsov, {\em On generically stable types in dependent theories}, 
J. Symb. Log., Volume 74-1, 2009.

\bibitem{HH} Y. Halevi and A. Hasson, {\em Strongly dependent ordered abelian groups and Henselian fields}, 
preprint, 2017, arXiv:1706.03376.

\end{thebibliography}
\end{document}